\documentclass[a4paper,10pt]{article}

\usepackage{graphicx}

\usepackage{amssymb, amsmath, amscd, amsfonts, amsthm}
\usepackage{colonequals}

\theoremstyle{plain}
\numberwithin{equation}{section} 
\newtheorem{theorem}[equation]{Theorem}
\newtheorem{corollary}[equation]{Corollary}
\newtheorem{lemma}[equation]{Lemma}
\newtheorem{proposition}[equation]{Proposition}

\theoremstyle{definition}
\newtheorem{definition}[equation]{Definition}
\newtheorem{remark}[equation]{Remark}
\newtheorem{example}[equation]{Example}
\newtheorem{question}[equation]{Question}
\newtheorem*{keywords}{keywords}

\newtheoremstyle{dotless}{}{}{}{}{\bfseries}{}{ }{}
\theoremstyle{dotless}
\newtheorem*{ccode}{Mathematics Subject Classification 2000:}

\newcommand{\nth}{\ensuremath{^{\textrm{th}}}}
\newcommand{\fst}{\ensuremath{^{\textrm{st}}}}
\newcommand{\from}{\colon} 

\newcommand{\calB}{\mathcal{B}}
\newcommand{\ZZ}{\mathbb{Z}}
\newcommand{\defeq}{\colonequals}
\newcommand{\isom}{\cong} 

\DeclareMathOperator{\inte}{int}
\DeclareMathOperator{\Mod}{Mod^+} 
\DeclareMathOperator{\lcm}{lcm}
\DeclareMathOperator{\Id}{Id}

\newcommand{\quotient}[2]{{\raisebox{0.2em}{$#1$}\left/\raisebox{-0.2em}{$#2$}\right.}}
\DeclareMathOperator{\Lk}{Lk} 
\DeclareMathOperator{\Conway}{\nabla} 

\title{The monodromies of homogeneous links}
\author{Mark Bell\\
Mathematics Institute\\
Zeeman Building\\
University of Warwick\\
Coventry, CV4 7AL,\\
\texttt{m.c.bell@warwick.ac.uk}}

\begin{document}

\maketitle

\begin{abstract}
We show that there are only finitely many homogeneous links whose Conway polynomial has any given degree. Using this we give an example of an inhomogeneous, fibred knot. Secondly, we show how to compute the monodromy of a homogeneous link complement from a homogeneous braid word representative.
\end{abstract}


\ccode{57M25}

\section{Introduction}

For a compact, orientable surface $S$, possibly with boundary, the \emph{mapping class group} of $S$ is the group of orientation preserving self-homeomorphisms of $S$ up to isotopy. This is denoted $\Mod(S)$. We use these mapping classes to build 3--manifolds from $S$ in the following way.

\begin{definition}
For $f \in \Mod(S)$ the \emph{mapping torus} $M_f$ is the 3--manifold
\[M_f \defeq \quotient{S \times [0, 1]}{(x, 1) \sim (\phi(x), 0)}\]
where $\phi \in f$ is any representative.
\end{definition}
It can be shown that, up to homeomorphism, this 3--manifold is independent of the choice of representative $\phi$. We say that $f$ is the \emph{monodromy} of $M_f$ and $S$ is its \emph{fibre}.

A mapping torus is entirely determined by its fibre and monodromy. Stallings determined exactly which 3--manifolds are mapping tori.

\begin{theorem}[{\cite[Theorem 2]{Stallings_Fibring}}]
\label{thrm:Stallings}
A compact, orientable 3--manifold $M$ is a mapping torus if and only if there exists an epimorphism $\phi \from \pi_1(M) \to \ZZ$ such that the kernel of $\phi$ is finitely generated. \qed
\end{theorem}

Additionally, Stallings determined that $G$, the kernel of the epimorphism $\phi$ in Theorem~\ref{thrm:Stallings}, is the fundamental group of the fibre \cite[Theorem 2]{Stallings_Fibring}. Hence the fibre of a mapping torus is uniquely determined, for example, by the rank of $G$ (and whether or not the manifold is closed when $G$ has rank zero). We consider the remaining problem of determining the monodromy:

\begin{question}
\label{qst:determine_monodromy}
For which classes of mapping tori can we give useful expressions for the monodromy?
\end{question}

In Section~\ref{sec:Monodromies} we give an answer to this in the case when the 3--manifold is the complement of a homogeneous link in $S^3$.

Homogeneous links arise as the closure of a homogeneous braid \cite[page 57]{Stallings_Constructions}. These links were also studied by Stallings who showed that their complements are all mapping tori \cite[Theorem 2]{Stallings_Constructions}. In Section~\ref{sec:Homogeneous} we sketch his proof, a key detail of which is a decomposition of the fibre surface under the Murasugi sum. This decomposition is essential to Section~\ref{sec:Monodromies}; we determine the monodromy on each piece and combine these together using a result of Gabai \cite[Corollary 1.4]{GabaiII}.

In Section~\ref{sec:Shift_Map} we introduce the shift maps and show that they can be used to simplify a braid whilst preserving its closure. The simplest braids are the non-weak braids. Dasbach and Mangum showed that the degree of the HOMFLY polynomial of homogeneous link is related to the underlying homogeneous braid \cite[Proposition 4.1.1]{DasbachMangum}. In Proposition~\ref{prop:homogeneous_degree_condition} we show a similar result for the degree of the Conway polynomial of a homogeneous link. Combining this with the class of non-weak braids we obtain that there are only finitely many homogeneous links whose Conway polynomial has any given degree in Theorem~\ref{thrm:homogeneous_degree_finite}. Similarly, in Theorem~\ref{thrm:homogeneous_genus_finite}, we obtain that there are only finitely many homogeneous knots of any given genus.

By enumerating these we compute all possible homogeneous links whose Conway polynomial has degree at most three. These are listed in Corollary~\ref{cor:small_degree_links}. By similar analysis we compute all possible homogeneous knots with genus at most two. These are listed in Corollary~\ref{cor:small_genus_links}.

From this classification we determine that the $8_{20}$ knot \cite[Appendix~C]{Rolfsen} is an inhomogeneous, fibred knot.

\section{Homogeneous link complements}
\label{sec:Homogeneous}

Recall that the set of braids on $n$ strands form a group under concatenation called the \emph{braid group}. This is denoted $B_n$ and has a standard presentation
\[B_n \isom \langle \sigma_1, \ldots, \sigma_{n-1} ~|~ \sigma_i \sigma_{i+1} \sigma_i = \sigma_{i+1} \sigma_i \sigma_{i+1},~\sigma_i \sigma_j = \sigma_j \sigma_j~\textrm{if}~|i-j| > 1 \rangle.\]
Here $\sigma_i$ corresponds to the braid in which the $i\nth$ strand passes under the $(i+1)\fst$ strand. We also denote the set of \emph{braid words} on $n$ strands by $\calB_n$; this is the Kleene closure of $\{\sigma_1, \ldots, \sigma_{n-1}, \sigma_1^{-1}, \ldots, \sigma_{n-1}^{-1}\}$. Thus $B_n$ is the equivalence classes of $\calB_n$ under the relations of the braid group and the relation that $\sigma_i \sigma_i^{-1} = \epsilon$. We denote the braid class of a braid word $w \in \calB_n$ by $[w] \in B_n$. 

\begin{definition}
A braid word $w = w_1 \cdots w_m \in \calB_n$ is \emph{homogeneous} if for each $i$, the generator $\sigma_i$ appears in $w$ if and only if $\sigma_i^{-1}$ does not \cite[page~57]{Stallings_Constructions}.
\end{definition}

For a homogeneous braid word $w = w_1 \cdots w_m \in \calB_n$ we denote the sign of the exponent with which $\sigma_i$ appears in $w$ by $\alpha(i)$. Additionally, we denote the strand index of $w_i$ by $x(i)$, that is \[w_i = \sigma_{x(i)}^{\alpha(x(i))}.\]

We may take the \emph{braid closure} of a braid $\sigma$ to obtain a link $\beta(\sigma)$. For ease of notation, for a braid word $w$ we abbreviate $\beta([w])$ to $\beta(w)$.

\begin{definition}
A link $K$ is \emph{homogeneous} if there exists a homogeneous braid word $w$ such that $\beta(w) = K$.
\end{definition}

Note that this definition is much stronger than Cromwell's definition of a link being homogeneous \cite[page 536]{Cromwell}.

Associated to a homogeneous braid word $w = w_1 \cdots w_m \in \calB_n$ is an oriented surface $S(w)$ embedded in $S^3$. This is obtained by connecting $n$ disks together via $m$ half-twisted bands corresponding to the $w_i$, see Figure~\ref{fig:fibre_surface}. This surface has the property that $\partial S(w) = \beta(w)$ and naturally decomposes under the Murasugi sum, see Figure~\ref{fig:fibre_surface_decompose}.

\begin{figure}[ht]
\begin{minipage}{0.45\linewidth}
\centering
\includegraphics[height=0.85\linewidth]{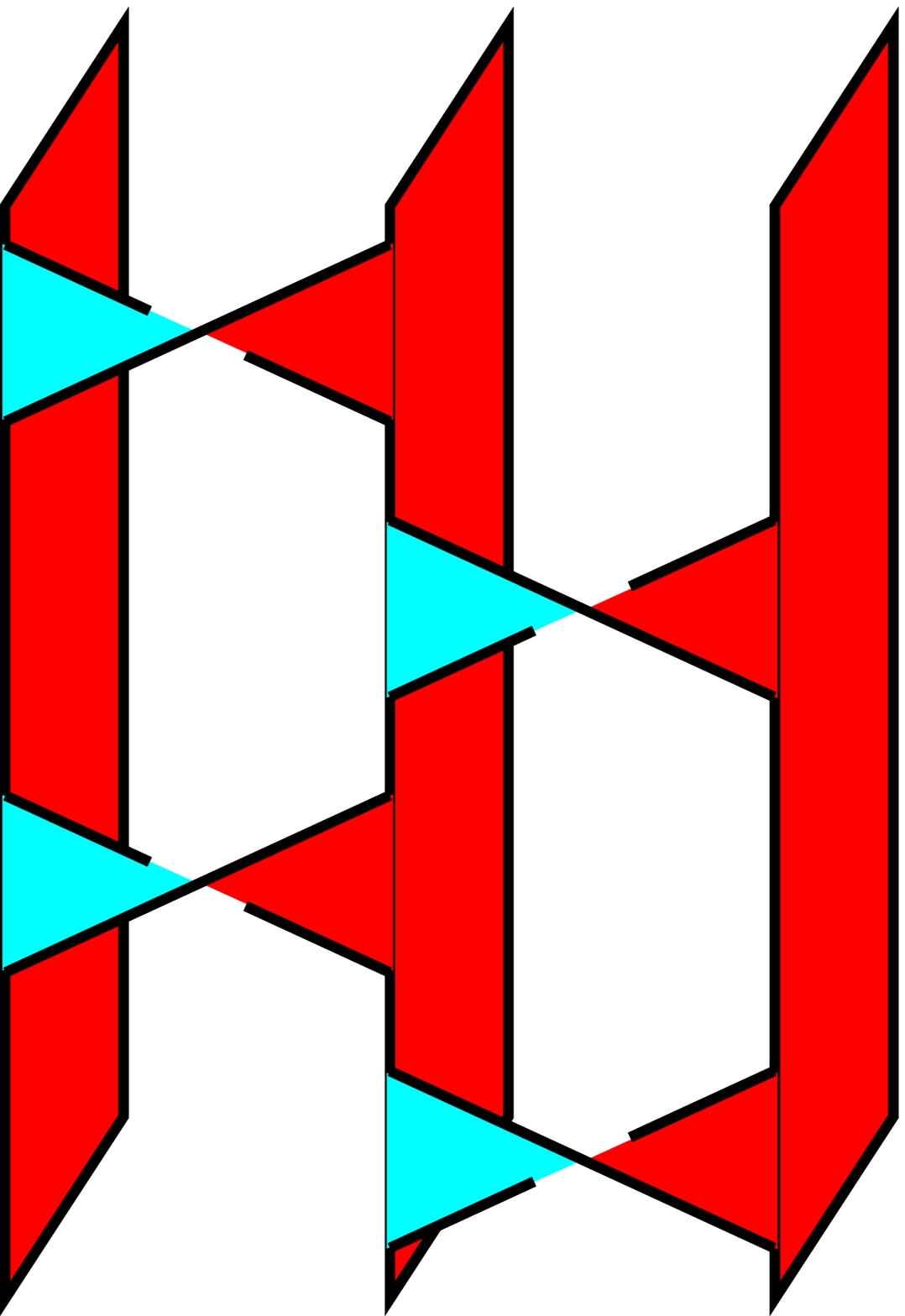}
\caption{The surface $S(\sigma_1 \sigma_2^{-1} \sigma_1 \sigma_2^{-1})$.}
\label{fig:fibre_surface}
\end{minipage}
\quad
\begin{minipage}{0.45\linewidth}
\centering
\includegraphics[height=0.85\linewidth]{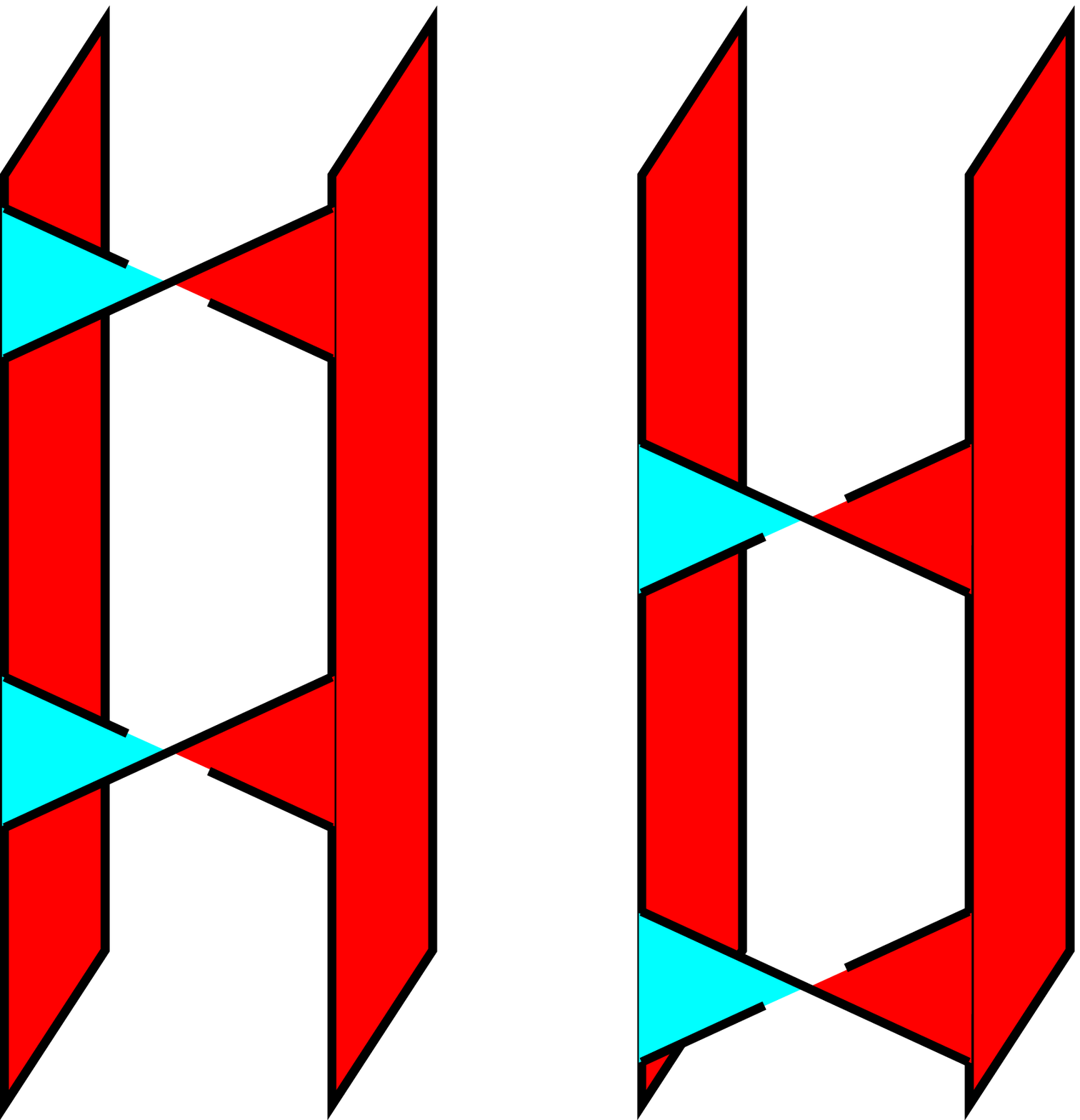}
\caption{The decomposition under Murasugi sum of $S(\sigma_1 \sigma_2^{-1} \sigma_1 \sigma_2^{-1})$.}
\label{fig:fibre_surface_decompose}
\end{minipage}
\end{figure}

\begin{definition}
Let $S$, $S_1$ and $S_2$ be compact, connected, oriented surfaces in $S^3$. If 
\begin{enumerate}
\item $S = S_1 \cup S_2$,
\item $S_1 \cap S_2 = D$ is a $2n$-gon whose edges are alternately in $\partial S_i$ and $\inte(S_i)$ for $i = 1, 2$, and
\item there is a sphere splitting $S^3$ into $B_1 \cup_\partial B_2$ such that
	\begin{enumerate}
	\item $S_1 \subseteq B_1$,
	\item $S_2 \subseteq B_2$, and
	\item $\partial B_1 \cap S_1 = D = \partial B_2 \cap S_2$
	\end{enumerate}
\end{enumerate}
then $S$ is the \emph{Murasugi sum} of $S_1$ and $S_2$ and we write $S = S_1 \#_M S_2$, see Figure~\ref{fig:Murasugi_Sum}.
\end{definition}

\begin{figure}[ht]
\centering
\includegraphics[height=0.3\linewidth]{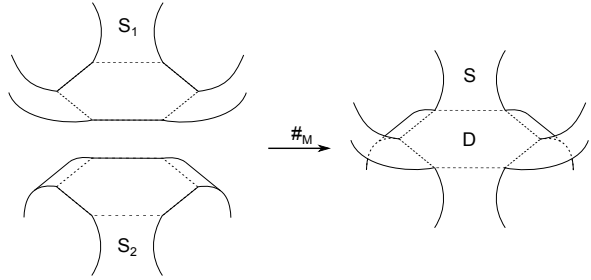}
\caption{The Murasugi sum of two surfaces when $n = 3$.}
\label{fig:Murasugi_Sum}
\end{figure}

As is standard, a link $K$ is \emph{fibred} if its complement $S^3 - n(K)$ is a mapping torus. 

\begin{theorem}[{\cite[Theorem 2]{Stallings_Constructions}}]
\label{thrm:homogeneous_fibred}
Every homogeneous link is fibred.
\end{theorem}

\begin{proof}
Let $K$ be a homogeneous link. Let $w = w_1 \cdots w_m \in \calB_n$ be a homogeneous braid word whose closure is $K$ and let
\[q_i \defeq | \{ j : w_j = \sigma_i^{\alpha(i)} \} |\]
count the number of occurrences of each generator in $w$. Let $v_i \defeq (\sigma_1^{\alpha(i)})^{q_i} \in \calB_2$ and $S_i \defeq S(v_i)$. Then $S(w)$ decomposes as the Murasugi sum: $S(w) = S_1 \#_M \cdots \#_M S_{n-1}$ and $K_i \defeq \partial S_i$ is the $(2, \alpha(i) q_i)$--torus link. A direct calculation using Theorem~\ref{thrm:Stallings} and the map
\[\gamma \mapsto \Lk(\gamma, K_i)\]
shows that each $K_i$ is fibred and its fibre is $S_i$. It is a Theorem of Stallings that if $K_1$ and $K_2$ fibred links with fibres $S_1$ and $S_2$ respectively and $S = S_1 \#_M S_2$ then $K = \partial S$ is a fibred link with fibre $S$ \cite[Theorem 1]{Stallings_Constructions}. Thus $\partial S(w) = \beta(w) = K$ is a fibred link with fibre $S(w)$.
\end{proof}

\section{The shift map}
\label{sec:Shift_Map}

The \emph{shift map} $s_i \from \calB_n \to \calB_{n-1}$ is the map which discards all $\sigma_i$ and $\sigma_i^{-1}$ twists in a braid word and moves later twists down by one strand. That is, if $w = w_1 \cdots w_m \in \calB_n$ is a braid word then $s_i(w) \defeq \widehat{w}_1 \cdots \widehat{w}_m \in \calB_{n-1}$ where 
\[
\widehat{w}_j \defeq 
\begin{cases}
\sigma_k^{\pm 1} & \textrm{if $w_j = \sigma_{k}^{\pm 1}$ and $k < i$,} \\
\epsilon & \textrm{if $w_j = \sigma_{k}^{\pm 1}$ and $k = i$,} \\
\sigma_{k-1}^{\pm 1} & \textrm{if $w_j = \sigma_k^{\pm 1}$ and $k > i$.}
\end{cases}
\]

\begin{example}
If $w = \sigma_1 \sigma_3 \sigma_5^{-1} \in \calB_6$ then 
\[s_1(w) = \sigma_2 \sigma_4^{-1} \quad \textrm{and} \quad s_4(w) = \sigma_1 \sigma_3 \sigma_4^{-1}.\]
\end{example}

In certain cases the closure of a braid word is invariant under the shift map.

\begin{definition}
A braid word $w = w_1 \cdots w_m \in \calB_n$ is \emph{$i$--weak} if $w_j$ is $\sigma_i$ or $\sigma_i^{-1}$ for exactly one value of $j$. A braid word is \emph{weak} if it is $i$--weak for some $i$.
\end{definition}

\begin{remark}
\label{rem:weak_words}
If $w = w_1 \cdots w_m \in \calB_n$ is a homogeneous braid word and $m < 2(n-1)$ then $w$ is weak.
\end{remark}

\begin{lemma}
\label{lem:shift_preserves_closure}
If $w$ is an $i$--weak braid word then $\beta(s_i(w)) = \beta(w)$.
\end{lemma}

\begin{proof}
Without loss of generality, assume that $\sigma_i$ occurs exactly once in $w$. Furthermore, by performing Markov moves of type I on $w$ we may assume that $\beta(w)$ is as shown in Figure~\ref{fig:Braid1}. Here the ends of strands are connected if and only if they lie on the same vertical line and $\sigma_i$ is the only visible crossing.
By isotoping the right hand side of the diagram, we obtain the link shown in Figure~\ref{fig:Braid2}. This is $\beta(s_i(\sigma))$.
\end{proof}

\begin{figure}[ht]
\begin{minipage}{0.45\linewidth}
\centering
\includegraphics[width=\linewidth,height=0.8\linewidth]{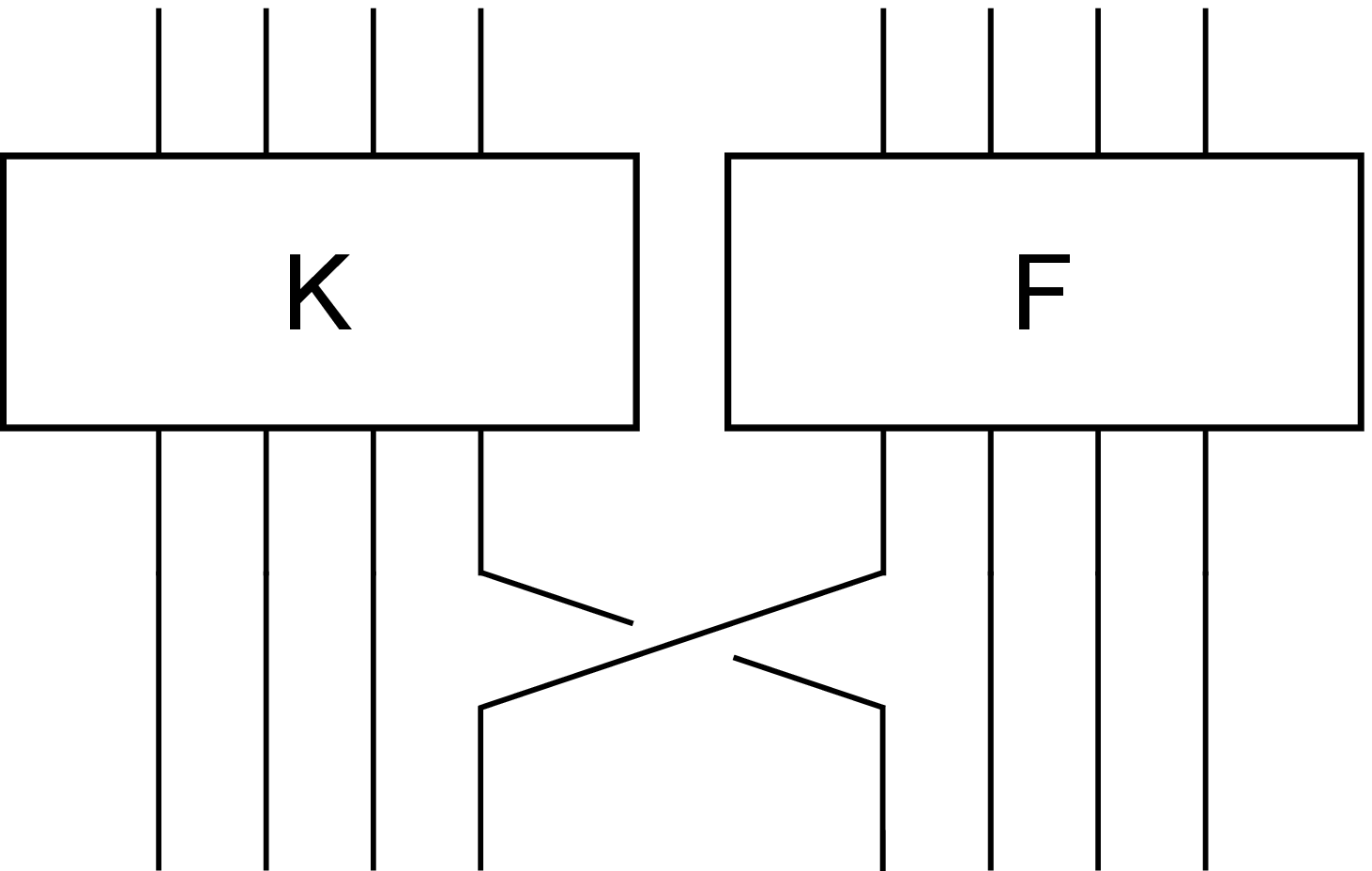}
\caption{The closure of an $i$--weak braid word.}
\label{fig:Braid1}
\end{minipage}
\quad
\begin{minipage}{0.45\linewidth}
\centering
\includegraphics[width=\linewidth,height=0.8\linewidth]{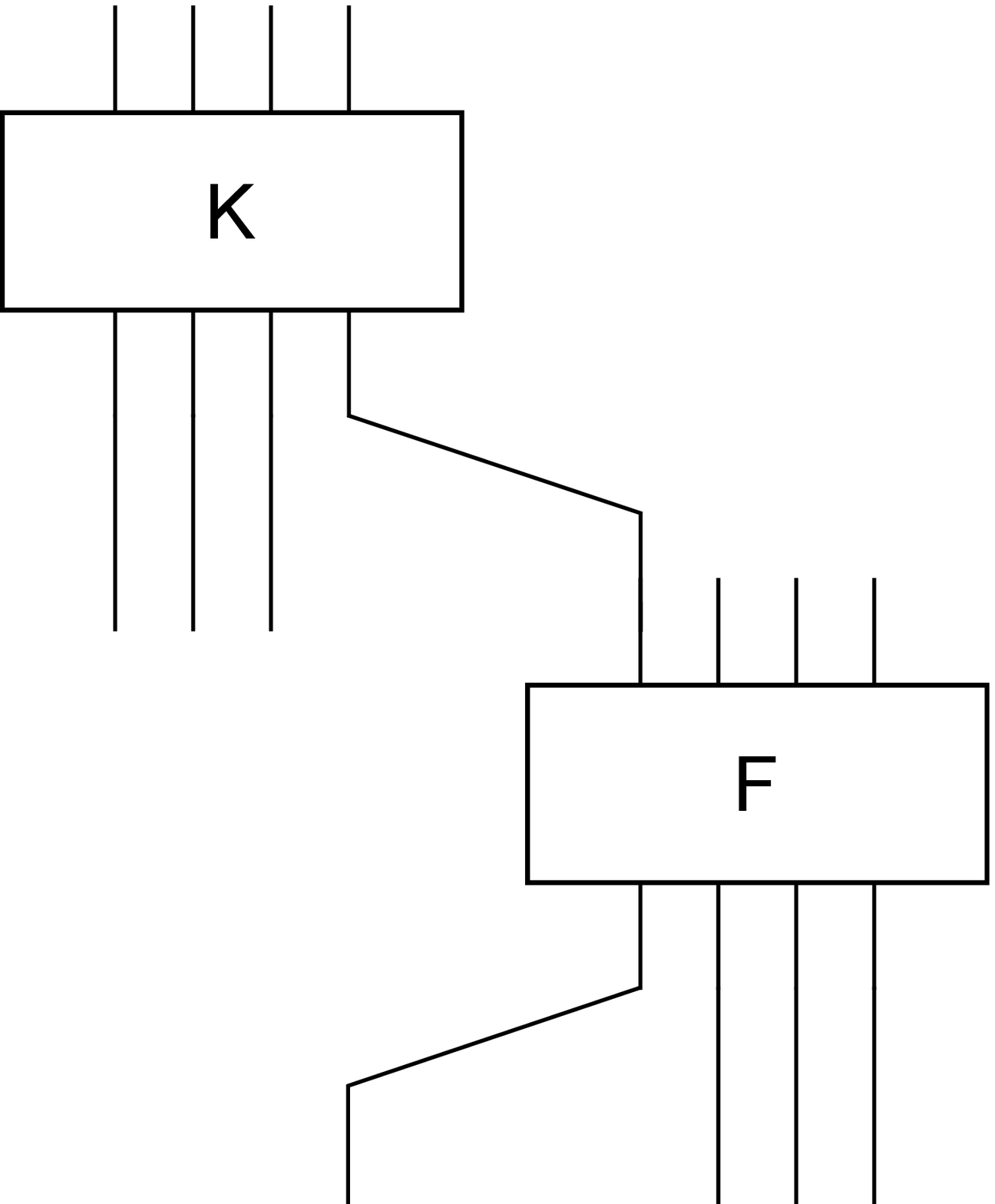}
\caption{The closure of an $i$--weak braid word after isotopy.}
\label{fig:Braid2}
\end{minipage}
\end{figure}

Hence we can use the shift maps to remove any weakness in a braid word representative of a link. Moreover, as the shift maps preserves homogeneity, we can also promote a homogeneous braid word to a homogeneous, non-weak braid word with the same closure.

\begin{proposition}
\label{prop:homo_non-weak_closure}
Every (homogeneous) link is the closure of a (homogeneous,) non-weak braid word. \qed
\end{proposition}

Recall that the \emph{Conway polynomial} \cite[page~105]{Kawauchi} $\Conway(K)(z)$ of a link $K$ is a polynomial invariant satisfying the skein relation
\[\Conway(K_+)(z) - \Conway(K_-)(z) = z \Conway(K_0)(z)\]
and the identity that
\[\Conway(\textrm{unknot})(z) = 1.\]
For a link $K$, we say that the \emph{degree} of $K$, $\deg(K)$, is the degree of the Conway polynomial of $K$; that is $\deg(K) \defeq \deg(\Conway(K))$. For ease of notation, for a braid word $w$ we abbreviate $\Conway(\beta(w))$ to $\Conway(w)$ and $\deg(\beta(w))$ to $\deg(w)$.

\begin{proposition}
\label{prop:homogeneous_degree_condition}
Let $w = w_1 \cdots w_m \in \calB_n$ be a homogeneous braid word. Then
\[\deg(w) = m - n + 1.\]
Moreover, the leading coefficient of $\Conway(w)$ is
\[\prod_{i=1}^n \alpha(i) \prod_{i=1}^m \alpha(x(i)).\]
\end{proposition}

\begin{proof}
First note that for any braid word $w$, when $j = i - 1$, the following identities hold:
\begin{eqnarray}
\Conway(w \sigma_i \sigma_i) &=& \Conway(w) + z \Conway(w \sigma_i) \label{eqn:Conway1} \\
\Conway(w \sigma_i^{-1} \sigma_i^{-1}) &=& \Conway(w) - z \Conway(w \sigma_i^{-1}) \label{eqn:Conway2} \\
\Conway(w \sigma_i \sigma_{j}^{-1} \sigma_i) &=& \Conway(w \sigma_{j}^{-1} \sigma_i \sigma_{j}^{-1}) + z \Conway(w \sigma_{j}^{-1} \sigma_i) + z \Conway(w \sigma_i \sigma_{j}^{-1}) \label{eqn:Conway3} \\
\Conway(w \sigma_i^{-1} \sigma_{j} \sigma_i^{-1}) &=& \Conway(w \sigma_{j} \sigma_i^{-1} \sigma_{j}) - z \Conway(w \sigma_{j} \sigma_i^{-1}) - z \Conway(w \sigma_i^{-1} \sigma_{j}) \label{eqn:Conway4} \\
\Conway(w \sigma_i \sigma_{j} \sigma_i) &=& \Conway(w \sigma_{j} \sigma_i \sigma_{j}) \label{eqn:Conway5} \\
\Conway(w \sigma_i^{-1} \sigma_{j}^{-1} \sigma_i^{-1}) &=& \Conway(w \sigma_{j}^{-1} \sigma_i^{-1} \sigma_{j}^{-1}) \label{eqn:Conway6}
\end{eqnarray}
The first four are direct results of the skein relation and the last two are direct results of the braid relation. Note that in each case if the braid on the left hand side is homogeneous then every braid used on the right hand side is too. Note also that if the result holds for the terms on the right hand side of an equation then it is also true for the term on the left hand side.

The \emph{complexity} of a homogeneous braid word $w = w_1 \cdots w_m \in \calB_n$ is the $(n-1)$--tuple $(q_{n-1}, \ldots q_1)$ where $q_j$ is the number of occurrences of $\sigma_{j}^{\alpha(j)}$ in $w$. We order the set of all homogeneous braid words short-lexicographically with respect to complexity. Note that in each of the above identities all of the terms on the right hand side of an equation are of lower complexity than of that on the left. We now proceed by induction on complexity. To deal with the base case, note that if the complexity of $w$ is $(1)$ then $w = \sigma_1^{\pm 1}$. In either case $\beta(w)$ is the unknot and so $\Conway(w) = 1$. Thus the result holds.

Now suppose that $q_j = 1$ for some $j$. Let $w' \defeq s_j(w)$ which is a braid word on one fewer strands with one fewer crossings. As $\beta(w) = \beta(w')$, by Lemma~\ref{lem:shift_preserves_closure}, $\deg(w) = \deg(w')$. By induction, as $w'$ is of lower complexity, 
\[\deg(w') = (m-1) - (n-1) + 1 = m - n + 1.\]
Moreover, as $q_j = 1$ there is a unique $k$ such that $x(k) = j$. By induction, the leading coefficient of $\Conway(w')$ is
\[ \prod_{\substack{i=1 \\ i \neq j}}^n \alpha(i) \prod_{\substack{i=1 \\ i \neq k}}^m \alpha(x(i)).\]
Hence as $\Conway(w) = \Conway(w')$ and $\alpha(j) \alpha(x(k)) = 1$, the leading coefficient of $\Conway(w)$ is
\[\prod_{i=1}^n \alpha(i) \prod_{i=1}^m \alpha(x(i)).\]
Thus the result holds. 

Thus, we now suppose that $q_j > 1$ for each $j$. In this case, by a discrete intermediate value theorem argument, $w$ can be written as:
\[w = w_1 \sigma_{j}^{\alpha(j)} w_2 \sigma_{j}^{\alpha(j)} w_3\]
such that $w_2$ contains no $\sigma_{j}^{\alpha(j)}$ terms, no $\sigma_{j+1}^{\alpha(j+1)}$ terms and at most one $\sigma_{j-1}^{\alpha(j-1)}$ term. Note that it is possible that $w_2 = \epsilon$.

If $\sigma_{j-1}^{\alpha(j-1)}$ does not occur in $w_2$ then let
\[w' \defeq w_3 w_1 w_2 \sigma_{j}^{\alpha(j)} \sigma_{j}^{\alpha(j)}.\] 
By far commutativity we have that $\beta(w) = \beta(w')$ and thus $\deg(w) = \deg(w')$. The result now follows by applying either equation \ref{eqn:Conway1} or \ref{eqn:Conway2} to $w'$.

Similarly, if $\sigma_{j-1}^{\alpha(j-1)}$ occurs exactly once in $w_2$ then
\[w_2 = w_4 \sigma_{j-1}^{\alpha(j-1)} w_5 \]
where $w_4$ and $w_5$ contain no $\sigma_{j-1}^{\alpha(j-1)}$ terms. Let
\[w' \defeq w_5 w_3 w_1 w_4 \sigma_{j}^{\alpha(j)} \sigma_{j-1}^{\alpha(j-1)} \sigma_{j}^{\alpha(j)}.\]
Then again $\beta(w) = \beta(w')$ and so $\deg(w) = \deg(w')$. The result now follows by applying either equation \ref{eqn:Conway3}, \ref{eqn:Conway4}, \ref{eqn:Conway5} or \ref{eqn:Conway6} to $w'$. 
\end{proof}

Dasbach and Mangum have shown a similar result for the HOMFLY polynomial of homogeneous links \cite[Proposition 4.1.1]{DasbachMangum}.

\begin{theorem}
\label{thrm:homogeneous_degree_finite}
For each $k \geq 0$ there are finitely many homogeneous links of degree $k$.
\end{theorem}

\begin{proof}
Let $w = w_1 \cdots w_m \in \calB_n$ be a homogeneous, non-weak braid word such that $\beta(w)$ is a link of degree $k$. As $w$ is a non-weak braid word, $m \geq 2(n-1)$ by Remark~\ref{rem:weak_words} and by Proposition~\ref{prop:homogeneous_degree_condition}
\[k = m - n + 1.\]
Combining these we obtain that:
\[2 \leq n \leq k + 1 \quad \textrm{and} \quad m = k + n - 1. \]
Hence, there are only finitely many homogeneous, non-weak braid words whose closure is a link of degree $k$ and so, by the homogeneous version of Proposition~\ref{prop:homo_non-weak_closure}, only finitely many such links.
\end{proof}

We can in fact provide a more exact bound. Let $p(k)$ be the number of homogeneous links of degree $k$. Note that if we fix the exponent of each generator of $\calB_n$ then there are $(n-1)^m$ possible braid words of length $m$. Hence there are at most $2^{n-1} (n-1)^m$ homogeneous braid words of length $m$ in $\calB_n$. Thus $p(0) = 1$ and when $k > 0$,
\[p(k) \leq \sum_{n = 1}^{k} 2^n n^{k+n}.\]
Hence $p(0) = 1$, $p(1) \leq 2$, $p(2) \leq 66$ and $p(3) \leq 5962$. A more careful count taking symmetries into account shows that $p(2) \leq 10$ and $p(3) \leq 22$. Explicitly checking these gives:

\begin{corollary}
\label{cor:small_degree_links}
Up to mirror-reflection, if a homogeneous link is:
\begin{itemize}
\item of degree zero then it is the unknot,
\item of degree one then it is the Hopf link,
\item of degree two then it is either: the trefoil, the figure-eight knot or the 3--component chain (that is $\beta(\sigma_1 \sigma_1 \sigma_2 \sigma_2 \sigma_3 \sigma_3)$), or
\item of degree three then it is the closure of either:
	\begin{itemize}
	\item $\sigma_1 \sigma_1 \sigma_1 \sigma_1$,
	\item $\sigma_1 \sigma_1 \sigma_1 \sigma_2 \sigma_2$,
	\item $\sigma_1 \sigma_1 \sigma_2 \sigma_2 \sigma_3 \sigma_3$,
	\item $\sigma_1 \sigma_1 \sigma_2 \sigma_3 \sigma_3 \sigma_2$,
	\item $\sigma_1 \sigma_1 \sigma_2 \sigma_3^{-1} \sigma_2 \sigma_3^{-1}$, or
	\item $\sigma_1 \sigma_2^{-1} \sigma_1 \sigma_3 \sigma_2^{-1} \sigma_3$. \qed
	\end{itemize}
\end{itemize}
\end{corollary}

Recall that the \emph{genus} of a knot $K$ is the genus of a minimal genus \emph{Seifert surface}; that is a compact, connected, orientable surface $S$ such that $\partial S = K$. As the degree of a knot is bounded above by twice its genus \cite[Theorem 7.2.1]{Cromwell_Knots_Links}, a result similar to Theorem~\ref{thrm:homogeneous_degree_finite} regarding genus immediately follows.

\begin{theorem}
\label{thrm:homogeneous_genus_finite}
For each $g \geq 0$ there are finitely many homogeneous knots of genus $g$. \qed
\end{theorem}

Note that the corresponding statement about homogeneous links is false. For every $q$ the $(2, 2q)$--torus link is homogeneous but has genus zero.

Theorem~\ref{thrm:homogeneous_genus_finite} can also be proven geometrically. If $w = w_1 \cdots w_m \in \calB_n$ is a homogeneous braid word and $\beta(w)$ is a knot then its fibre is a minimal genus Seifert surface \cite[Proposition 4.1.10]{Kawauchi}. Hence the genus of $\beta(w)$ is $(1 + m - n) / 2$.

Again, we can provide a more exact bound. Let $n(g)$ be the number of homogeneous knots of genus $g$. Then, when $g > 0$, by the same arguement as before,
\[n(g) \leq \sum_{n=1}^{2g} 2^{n} n^{2g+n}.\]
By explicitly checking all homogeneous knots of genus at most two, we obtain the following corollary.

\begin{corollary}
\label{cor:small_genus_links}
Up to mirror-reflection, if a homogeneous knot is:
\begin{itemize}
\item of genus zero then it is the unknot,
\item of genus one then it is either: the trefoil or the figure-eight knot, or
\item of genus two then it is either: the $5_1$, $6_2$, $6_3$, $7_6$, $7_7$, $8_{12}$, $3_1 \# 3_1$, $3_1 \# r(3_1)$, $3_1 \# 4_1$ or $4_1 \# 4_1$ knot \cite[Appendix~C]{Rolfsen}. Here $r(3_1)$ is the reflection of the $3_1$ knot. \qed
\end{itemize}
\end{corollary}

\begin{corollary}
\label{cor:inhomogeneous_fibred_knots}
The $8_{20}$ knot is an inhomogeneous, fibred knot.
\end{corollary}

\begin{proof}
Stallings showed that the $8_{20}$ knot \cite[Appendix~C]{Rolfsen}, shown in Figure~\ref{fig:8_20}, is fibred \cite[pages~58~--~59]{Stallings_Constructions}. It is a genus two knot and so by Corollary~\ref{cor:small_degree_links}, it is inhomogeneous.
\end{proof}

\begin{figure}[ht]
\centering
\includegraphics[height=0.42\linewidth]{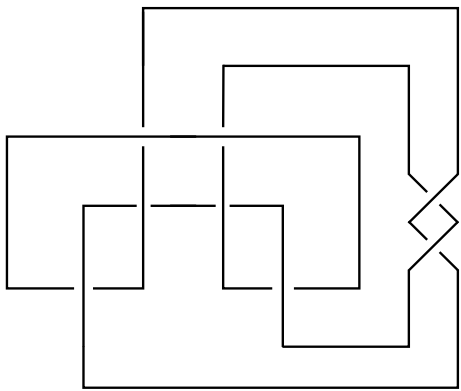}
\caption{The $8_{20}$ knot.}
\label{fig:8_20}
\end{figure}

\section{Monodromies}
\label{sec:Monodromies}

We now give an answer to Question~\ref{qst:determine_monodromy} for homogeneous link complements. A \emph{loop} on a surface $S$ is the image of a smooth embedding of $S^1$ into $S$. A loop inherits an orientation from an orientation on $S^1$. For two loops to be equal both their images and orientations must agree.

\begin{definition}
A collection of loops $\{\gamma_1, \ldots, \gamma_n\}$ on $S$ is in \emph{general position} if:
\begin{itemize}
	\item when $i$ and $j$ are distinct $\gamma_i \cap \gamma_j$ is a finite set,
	\item when $i$ and $j$ are distinct $\gamma_i \pitchfork \gamma_j$, and
	\item when $i$, $j$ and $k$ are distinct $\gamma_i \cap \gamma_j \cap \gamma_k = \emptyset$.
\end{itemize}
\end{definition}

\begin{definition}
A collection of loops $\{\gamma_1, \ldots, \gamma_n\}$ on $S$ is in \emph{minimal position} if $S - (\gamma_i \cup \gamma_j)$ does not contain any bigons for any $i$ and $j$.
\end{definition}

Note that we can always isotope any collection of loops on $S$ such that it is in both general and minimal position simultaneously.

\begin{definition}
An annulus $A$ embedded in $S$ is \emph{peripheral} if at least one boundary component of $A$ is a boundary component of $S$.
\end{definition}

\begin{definition}
A collection of loops $\{\gamma_1, \ldots, \gamma_n\}$ on $S$ \emph{fill} if $S - \bigcup_{i} \gamma_i$ is a disjoint collection of disks and peripheral annuli.
\end{definition}

\begin{definition}
A collection of loops $\{\gamma_1, \ldots, \gamma_n\}$ on $S$ is \emph{triangle-free} if when $i$, $j$ and $k$ are distinct at least one of $\gamma_i \cap \gamma_j$, $\gamma_i \cap \gamma_k$ or $\gamma_j \cap \gamma_k$ is empty
\end{definition}

The Alexander method says that a mapping class is uniquely determined by the image of a suitable collection of loops.

\begin{theorem}[{The Alexander Method \cite[Proposition 2.8]{FarbMargalit}}]
\label{thrm:Alexander}
Let $S$ be a surface and $\Gamma = \{\gamma_1, \ldots, \gamma_n\}$ a collection of loops on $S$ such that:
\begin{itemize}
\item $\Gamma$ is in general position, 
\item $\Gamma$ is in minimal position,
\item $\Gamma$ is triangle-free,
\item $\Gamma$ fills $S$, and
\item when $i$ and $j$ are distinct $\gamma_i$ and $\gamma_j$ are non-isotopic.
\end{itemize}
If two mapping classes $f$ and $g$ have representatives $\phi$ and $\psi$ respectively such that $\phi(\gamma_i) = \psi(\gamma_i)$ for each $i$ then $f = g$. \qed
\end{theorem}

We begin by considering the case when $K$ is the $(2, q)$--torus link. For ease of argument we assume that $q > 0$, although the $q < 0$ case follows analogously. The link $K$ is homogeneous; $w \defeq \sigma_1^q \in \calB_2$ is a homogeneous braid word whose closure is $K$. Let $M$ be the complement of $K$. Then $M$ is a mapping torus with fibre $S(w)$ and monodromy $h$. Here $S(w)$ can be seen as two disks connected via $q$ half-twisted bands as shown in Figure~\ref{fig:Torus_SS}.

\begin{figure}[ht]
\centering
\includegraphics[width=0.4\linewidth,height=0.4\linewidth]{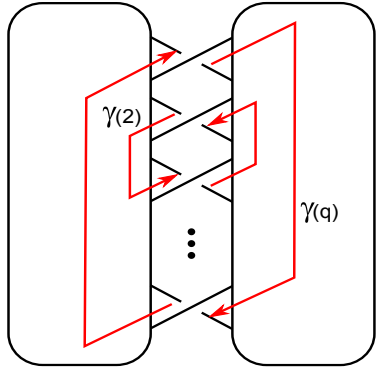}
\caption{The $(2, q)$--torus link.}
\label{fig:Torus_SS}
\end{figure}

For ease of notation, all addition and subtraction is done modulo $q$ throughout. Let $\gamma(i)$ be the loop on $S(w)$ starting on the right hand disk, passing to the left hand disk via the $i\nth$ half-twisted band and returning via the $(i+1)\fst$ one. By pushing this loop out of the positive side of $S(w)$ and returning it to $S(w)$ via the negative side, we see that $h(\gamma(i))$ is $\gamma(i-1)$ with the reversed orientation. Let $\Gamma \defeq \{\gamma(i)\}$ and note that this collection of curves satisfies the hypotheses of Theorem~\ref{thrm:Alexander}. Now consider
\[T(\gamma(q-1)) \circ \cdots \circ T(\gamma(1))\]
where $T(\gamma)$ is a left Dehn twist about $\gamma$. Then, up to isotopy,
\[T(\gamma(q-1)) \circ \cdots \circ T(\gamma(1)) (\gamma(i)) = h(\gamma(i))\]
and so by Theorem~\ref{thrm:Alexander}
\[h = T(\gamma(q-1)) \circ \cdots \circ T(\gamma(1)).\]
Note that if $q$ was negative then the monodromy would be 
\[h = (T(\gamma(q-1)) \circ \cdots \circ T(\gamma(1)))^{-1}.\]

\begin{corollary}
The monodromy of the $(2, q)$--torus link has order $\lcm(2, q)$. \qed
\end{corollary}

In the more general case, we use the following Theorem of Gabai who showed that Murasugi sums interact nicely with monodromies.

\begin{theorem}[{\cite[Corollary 1.4]{GabaiII}}]
\label{thrm:Murasugi_monodromy}
Suppose that $S = S_1 \#_M S_2$, $K = \partial S$ and $K_i = \partial S_i$. If for each $i$, $K_i$ is a fibred link with fibre $S_i$ and monodromy $h_i$ and $h_i|_{\partial S_i} = \Id$ then $K$ is a fibred link with fibre $S$ and 
\[h = h_1' \circ h_2',\]
where 
\[
h_i'(x) \defeq 
\begin{cases}
h_i(x) & \textrm{if $x \in S_i$,} \\
\Id(x) & \textrm{otherwise,}
\end{cases}
\]
is its monodromy. \qed
\end{theorem}

Now let $K$ be a homogeneous link. Let $w = w_1 \cdots w_m \in \calB_n$ be a homogeneous, non-weak braid word whose closure is $K$. Let $M$ be the complement of $K$ with fibre $S(w)$ and monodromy $h$. As before we think of $S(w)$ as $n$ disks connected via $m$ half-twisted bands corresponding to the $w_i$. We index the $j\nth$ half-twisted band connecting from the $i\nth$ disk to the $(i+1)\fst$ disk by $b(i,j)$. Let $\gamma(i,j)$ be the loop on $S(w)$ which starts on the $i\nth$ disk, goes to the $(i+1)\fst$ disk via $b(i,j)$ and returns via $b(i,j+1)$.

Recall that
\[q_i \defeq | \{ j : w_j = \sigma_i^{\alpha(j)} \} |\]
is the number of occurrences of each generator in $w$ and $v_i \defeq (\sigma_1^{\alpha(i)})^{q_i} \in \calB_2$. Let $S_i = S(v_i)$ then $S(w)$ decomposes as the Murasugi sum: $S(w) = S_1 \#_M \cdots \#_M S_{n-1}$. Let $K_i \defeq \partial S_i$, which is the $(2, \alpha(i) q_i)$--torus link with fibre $S_i$. 

Let
\[h'_i \defeq (T(\gamma(i,q_i - 1)) \circ \cdots \circ T(\gamma(i,1)))^{\alpha(i)}. \]
Then $h'_i|_{S_i}$ is the monodromy of $K_i$ and $h'_i|_{S(w) - S_i} = \Id$. Thus by Theorem~\ref{thrm:Murasugi_monodromy}, $K$ has monodromy
\[h \defeq h_1' \circ \cdots \circ h_{n-1}'.\]

\begin{corollary}
If $w = w_1 \cdots w_m$ is a homogeneous braid word and $M \defeq S^3 - n(\beta(w))$ then the monodromy of $M$ can be written as a compositition of at most $m$ Dehn twists. \qed
\end{corollary}

\section*{Acknowledgments}

The author wishes to thank Saul Schleimer and Jessica Banks for helpful suggestions and corrections.

The author is supported by an Engineering and Physical Sciences Research Council (EPSRC) studentship.

\bibliographystyle{plain}
\bibliography{bibliography}

\end{document}